\documentclass[reqno,11pt]{amsart}
\usepackage{amssymb}
\usepackage{amsmath}
\usepackage{amsthm}
\usepackage{pgf}
\usepackage{color}
\usepackage{mathrsfs}
\usepackage{graphicx}   
\usepackage{hyperref}
\usepackage{graphicx}
\usepackage{tikz}
\usepackage{pgf}

\newtheorem{theorem}{Theorem}[section]

\newtheorem{proposition}[theorem]{Proposition}

\newtheorem{remark}[theorem]{Remark}

\newcommand{\Nz}{\mathbb{N}}

\newcommand{\epsi}{\varepsilon}
\newcommand{\eps}{\varepsilon}

\def\cL{\mathscr{L}}
\def\cA{\mathscr{A}}
\def\cU{\mathscr{U}}
\def\cV{\mathscr{V}}
\def\cF{\mathscr{F}}

\def\eps{\varepsilon}
\def\cP{\mathscr{P}}

\renewcommand{\P}{\mathbb{P}}
\renewcommand{\d}{{\mathrm d}}
\newcommand{\ds}{{\mathrm d} s}

\def\to{\rightarrow}
\def\wto{\rightharpoonup}
\def\wstarto{\stackrel{*}{\rightharpoonup}}
\def\embed{\hookrightarrow}
\newcommand{\weakto}{\rightharpoonup}
 \newcommand{\E}{{\mathbb E}}

\begin{document}
\title[Energy-Dissipation Principle for SPDEs]
{The Energy-Dissipation Principle for \\  stochastic parabolic equations}

\author{Luca Scarpa}
\address[Luca Scarpa]{Department of Mathematics, Politecnico di Milano, 
Via E.~Bonardi 9, 20133 Milano, Italy.}
\email{luca.scarpa@polimi.it}
\urladdr{http://www.mat.univie.ac.at/$\sim$scarpa}

\author{Ulisse Stefanelli}
\address[Ulisse Stefanelli]{Faculty of Mathematics,  University of Vienna, Oskar-Morgenstern-Platz 1, A-1090 Vienna, Austria,
Vienna Research Platform on Accelerating Photoreaction Discovery,
University of Vienna, W\"ahringer Str. 17, 1090 Vienna, Austria, $\&$ Istituto di Matematica Applicata e Tecnologie Informatiche ``E. Magenes'' - CNR, v. Ferrata 1, I-27100 Pavia, Italy}
\email{ulisse.stefanelli@univie.ac.at}
\urladdr{http://www.mat.univie.ac.at/$\sim$stefanelli}



\begin{abstract} 
The Energy-Dissipation Principle provides a variational tool
for the analysis of parabolic evolution problems: solutions are
characterized as so-called null-minimizers of a global functional on
entire trajectories.  This variational technique allows for applying the general
results of the calculus of variations to the underlying differential
problem and has been successfully applied in a variety
of deterministic cases, ranging from doubly nonlinear flows to curves
of maximal slope in metric spaces. The aim of this note is to extend
the Energy-Dissipation Principle  to stochastic parabolic
evolution equations. Applications to stability and optimal control are
also presented.
\end{abstract}

\maketitle

%%%%%%%%%%%%%%%%%%%%%%%%%%%%%%%%%%%%%%%%%%%%%%%%
\section{Introduction}
\label{sec:intro}

This note is concerned with a global variational approach to the 
Cauchy problem for the  abstract
stochastic parabolic evolution equation
\begin{equation}\label{eq:0}
  \d u + \partial \phi(u)\,\d t \ni F(\cdot,u)\,\d t + G(\cdot,u)\,\d
  W, \quad  u(0)=u^0. 
\end{equation} 
The trajectory $u: \Omega \times [0,T] \to H$ is defined on the
stochastic basis $(\Omega,\cF,(\cF_t)_{t\in[0,T]},\P)$ and the bounded
time interval $[0,T]$ and takes values in  the  separable Hilbert space
$H$. The functional $\phi$ is assumed to be convex and lower semicontinuous, the
nonlinearities $F$ and $G$ are taken to be suitably smooth, and $W$ is 
a cylindrical Wiener process on a second separable Hilbert space
$U$. More precisely, solutions $u$ of equation \eqref{eq:0} are asked to be {\it It\^o
processes} of the form 
\begin{equation}\label{ito_process}
  u(t)=u^d(t) + \int_0^tu^s(s)\,\d W(s), \quad t\in[0,T],
\end{equation}
where $u^d$ is an absolutely continuous process
and $u^s$ is a $W$-stochastically integrable process. In particular,
we look for solutions $u$ of equation \eqref{eq:0} in the space 
$\cU$ consisting of all  It\^o
processes of the form \eqref{ito_process} with $u^d \in L^2_\cP (\Omega;
H^1(0,T;H))$ and 
$u^s\in L^2_\cP (\Omega;
L^2(0,T; \cL^2(U,V)))$.  
 Here,   $ \cL^2(U,V)$
indicates the set of Hilbert-Schmidt operators from $U$ to $V$, 
where 
$V$ is a separable reflexive Banach space, densely and compactly
embedded into $H$.

Existence, uniqueness, and continuous dependence on the initial 
data  for stochastic evolution problems in the form \eqref{eq:0}
are  addressed  within  the classical variational theory 
by {Pardoux} \cite{Pard0,Pard} and {Krylov
\&  Rozovski\u \i},\cite{kr} in the sense of analytically weak
 or {\it martingale}  solutions:
we refer also to the  monographs \cite{dapratozab,pr}
for a general presentation. In the context of analytically strong solutions, 
existence for stochastic equations in the subdifferential form \eqref{eq:0}
has been obtained by {Gess} \cite{Gess-strong}.
Well-posedness results in a weak sense
have also been obtained 
under more general conditions in  the monograph \cite{LiuRo} and
in the  papers 
\cite{mar-scar-diss,mar-scar-ref, mar-scar-div, scar-div}.

Following the seminal remarks by De Giorgi \cite{Ambrosio}, one can variationally characterize solutions of the Cauchy
problem \eqref{eq:0} in terms of trajectories minimizing the 
{\it Energy-Dissipation-Principle (EDP)} functional $I:\cU \to [0,\infty]$
defined as

\begin{align}
 &I (u) = 
          \E\phi(u(T)) - \E \phi(u(0)) + \frac12 \E \int_0^T \|
  \partial_t u^d\|^2  \, \ds    + \frac12 \E \int_0^T \| \partial
   \phi(u)-F(\cdot,u)  \|^2 \, \ds \nonumber\\
  &\qquad- \E \int_0^T
  (\partial_t u^d, F(\cdot,u)) \, \ds \ - \frac12 \E \int_0^T {\rm Tr}_H\, L(u) \, \ds 
    \nonumber\\
    &\qquad+ 2C_\phi\E\int_0^T \| u^s -
    G(\cdot,u)\|^2_{\cL^2(U,V)}  \, \ds   +\E\|u(0)- u^0\|^2_V %+ |\E \phi(u(0)) - \E \phi(u^0)|
\label{eq:EDP}
\end{align} 

\noindent if $u \in C([0,T];L^2(\Omega,\cF;V))$  and $I(u)=\infty$
    otherwise.
    Here, $L(u) := u^s (u^s )^* D_{\mathcal
  G}\partial \phi(u)$  %  \in \cL^1(V,V)$ and $I(u)=\infty$ otherwise.
and  $u^d,\, u^s$ are associated to $u$ via the decomposition
\eqref{ito_process}. The symbols $(\cdot, \cdot)$ and $\| \cdot\|$ stand for
the scalar product and the norm in $H$, respectively, and $ \partial
\phi(u)$ denotes the subdifferential of $\phi$, here assumed to be
Gateaux-differentiable  from $V$ to $V^*$.  The constant $ C_{\phi}>0$ depends on $\phi$
and is defined in \eqref{ass:phi2} below. Within our assumption setting, we will have that $
L(u) \in \cL^1(V,V)$, where the latter is the space of {\it  trace-class} operators from $V$ to $V$. The symbol ${\rm Tr}_H$
hence denotes the {\it trace} of the operator with respect to an orthonormal
system of $H$ in $V$.
The fact that $I$ takes nonnegative values hinges on the validity of an It\^o formula for $\phi$,
see Proposition~\ref{prop:ito} below. 

The focus of this note is to discuss the equivalence of solutions
of equation \eqref{eq:0} and {\it null-minimizers} of the EDP functional
$I$. Under general assumptions on $\phi$, $F$, and $G$, our main
result, Theorem \ref{thm:main}, states that 
$$u \ \ \text{solves \eqref{eq:0}} \ \ \Leftrightarrow \  \ 0 =
I(u) = \min_{\cU}I.$$
The core of this characterization resides on the nature of the EDP
functional $I$, which in the present setting corresponds to the
{\it squared residual} of the system
$$\partial_t u^d + \partial \phi(u) = F(\cdot, u), \quad u^s =
G(\cdot,u),\quad u(0)=u^0,$$
as illustrated in Proposition \ref{prop:equiv} below. The approach in
\eqref{eq:EDP} is however more general and can be adapted in Banach
spaces and doubly nonlinear problems  as well,  see Remark \ref{w10}
below. 

The residual nature of the EDP functional entails that the EDP variational
principle $0 = I(u) = \min_{\cU}I$ is  not  a mere minimization
problem, for one is asked to check that the minimum is actually
$0$, motivating the use of the term {\it
  null-minimization}. This issue is not uncommon for global
variational approaches and can be traced back to celebrated {\it Brezis-Ekeland-Nayroles}
principle
\cite{Brezis-Ekeland76b,Brezis-Ekeland76,Nayroles76,Nayroles76b}. 
In the  current stochastic case, the existence of a unique
null-minimizer follows from the well-posedness of the differential
problem \eqref{eq:0}. Still, minimization cannot be tackled directly,
for the functional $I$ shows some limited semicontinuity properties,  see
Section \ref{sec:coercivity}.
Apart from the case when $\Omega$ is atomic, this prevents us from providing an alternative existence theory for
problem \eqref{eq:0}. On the other hand,   we make use of the  EDP
characterization  for proving  
stability of the Cauchy problem for equation \eqref{eq:0} under
perturbations of the data $(u^0,\phi, F,G)$ in Section
\ref{sec:stability} and for discussing  the  penalization of an
optimal control problem constrained to  \eqref{eq:0}  in Section
\ref{sec:control}.

Before moving on, let us mention two alternative global variational
principles for SPDEs of the class \eqref{eq:0}. The mentioned
Brezis-Ekeland-Nayroles principle has been  indeed  extended to the stochastic
case. Following some specific application   in
\cite{Barbu1,Barbu2,Barbu3}, a general theory has been presented by
Barbu \& R\"ockner \cite{Barbu6,Barbu5,Barbu7} in the linear multiplicative
case and by Boroushaki \& Ghoussoub \cite{Boroushaki}   in the nonlinear  multiplicative
case. In our notation, the stochastic Brezis-Ekeland-Nayroles
functional from \cite{Boroushaki} reads
\begin{align*}
  u \ \mapsto \ &\E \int_0^T \left(\phi(u) + \phi^*(F(\cdot,u)-\partial_tu^d) -
  (F(\cdot,u)-\partial_tu^d,u)\right)\, \d s \\
  &+ \frac12\E\int_0^T \| u^s -
  G(\cdot,u)\|^2_{\cL^2(U,V)}\, \ds + \E \|u(0) -u^0\|^2. 
\end{align*}
Here, $\phi^*$ stands for the Legendre conjugate of $\phi$. Recall that $\phi(u)+\phi(v) \geq (v,u) $ for all $u,\,v \in H$ and
that equality holds if and only if $v \in \partial \phi(u)$. Hence, a
null-minimizer  of the latter necessarily solves the Cauchy problem
for \eqref{eq:0}. By resorting to the far-reaching theory of anti-self
dual Lagrangians, the existence of null-minimizers  of the Brezis-Ekeland-Nayroles
functional has been  ascertained  in 
\cite{Boroushaki}.

In the additive case, a different global variational approach to \eqref{eq:0} is in
\cite{ScarStef-WED}, where the {\it
  Weighted-Energy-Dissipation} functional
\begin{align*}
  u \ \mapsto \ & \E\displaystyle\int_0^T  e^{{-s}/{\eps}}
  \left(\frac\eps2\|\partial_t u^d\|^2 + \phi(u)  
     - (F,u)
 \right)\, \ds  + \E\displaystyle\int_0^T 
 e^{{-s}/{\eps}}\frac12\|u^s- G\|^2_{\cL^2(U,V)}\, \ds
\end{align*}
 is investigated. 
This  {\it strictly convex} functional admits a unique minimizer  $u_\epsi$ over trajectories with
given initial value $u^0$. At all levels $\epsi>0$, such
minimizers solve an
elliptic-in-time regularization of equation \eqref{eq:0},
complemented by an extra Neumann boundary condition at the final time
$T$. In particular, the minimization of the
Weighted-Energy-Dissipation functional corresponds to a {\it
  noncausal}  differential  problem.
As $\epsi \to 0$ one can prove \cite{ScarStef-WED} that $u_\epsi $
converge to  the 
solution to the Cauchy problem \eqref{eq:0}. In particular,
causality is restored in the limit. 

Compared with the Brezis-Ekeland-Nayroles approach, the null-minimization of the EDP
functional is a priori not restricted to the case of a convex 
$\phi$ (although we limit ourselves to convex $\phi$ in this note, for
simplicity) and
is easily adapted to more nonlinear situations, see Remark
\ref{w10} below. With respect to the Weighted-Energy-Dissipation
approach, the null-minimization of the EDP functional does not require
to take the extra limit $\epsi \to 0$ and is causal. It is hence
better suited to discuss convergence issues.

We collect notation, assumptions, and the statement of the
characterization, i.e., Theorem \ref{thm:main}, in Section
\ref{sec:setting}. Section \ref{sec:proof} is devoted to the proof of a
generalized It\^o formula, which is crucial for studying the EDP
functional and brings to the proof of the characterization. We discuss
in Section \ref{sec:coercivity}   the coercivity and the restricted 
lower-semicontinuity of the EDP functional, as well as the fact
that minimizers of $I$ are actually null-minimizers. We then  obtain a stability result with respect to data
perturbations in Section
\ref{sec:stability}. Eventually, in Section~\ref{sec:control} we
discuss a general penalization procedure for  an  optimal control
problem based on~\eqref{eq:0}.

%%%%%%%%%%%%%%%%%%%%%%%%%%%%%%%%%%%%%%%%%%%%%%%%
\section{Setting and statement}
\label{sec:setting}

In preparation of the statement of our main result, let us collect
here the assumptions on spaces and nonlinearities, which will be
tacitly assumed throughout the paper.

Let  $(\Omega,\cF,(\cF_t)_{t\in[0,T]},\P)$ be a stochastic basis,
   with right-continuous and saturated filtration
   $(\cF_t)_{t\in[0,T]}$, $H$ and $U$ be separable
    Hilbert spaces, and $W$ be  a 
    cylindrical Wiener process on $U$. Moreover, let $V$ be a separable and reflexive Banach
    space, with $V
    \subset H$  densely  and compactly, so that $V \subset
    H \subset V^*$ (dual) is a Gelfand triplet.  We recall that
    the  symbols $(\cdot, \cdot)$
    and $\| \cdot \|$ denote the scalar product and the norm in
    $H$.  Moreover,  $\langle \cdot, \cdot\rangle$ stands for the duality product
    between $V^*$ and $V$.
    The norm in any other Banach space $E$ will be denoted by $\|
    \cdot \|_E$.

   In the following, we use the classical notation $\cL(V,H)$,
   $\cL^2(U,V)$, and $\cL^1(V,V)$ to indicate the space of linear and
   continuous operators from $V$ to $H$, the space of Hilbert-Schmidt
   operators from $U$ to $V$, and the space of trace-class operators
   on $V$, respectively. The symbols $\cL_s(V,H)$ and $\cL_w(V,H)$
   indicate that the space $\cL(V,H)$ is endowed with the so-called
   {\it strong}, resp. {\it weak} operator topology. For all $L\in \cL^1(V,V)$ we denote by
   ${\rm Tr}_H \, L$ the {\it trace} of the operator with respect to an orthonormal
   system $(e_k)_{k\in \Nz}$ of $H$ contained in $V$, namely,
   $${\rm Tr}_H \, L = \sum_{i=1}^\infty (Le_k,e_k).$$

% Finally, we ask for a family
%     of regularizing linear operators $(R_\lambda)_{\lambda>0}\subset\cL(H,V)$ such that 
% $R_\lambda\to {\rm id}$ in $\cL_s(V,V)$ (strong operator topology) as $\lambda\searrow0$. Such a
% family $R_\lambda$ can readily be found 
% when $V$ is a Sobolev space.

We denote by $\cP$ the progressive $\sigma$-algebra on
$\Omega\times[0,T]$ and write $L^s(\Omega; E)$ and $L^s(0,T; E)$
for the spaces of strongly measurable Bochner-integrable $E$-valued functions 
on $\Omega$ and $(0,T)$, for all
$s\in[1,\infty]$ and a  Banach space $E$.
For $s,r\in[1, \infty)$ we use the symbol $L^s_\cP(\Omega;L^r(0,T; E))$ to indicate
that measurability is intended with respect to the progressive
$\sigma$-algebra $\cP$.

In the following, we will make use of the space $\cU$ of  It\^o
processes given by
\begin{align*}
  \cU=\{&u \in L^2(\Omega;C([0,T];H))  \, : \, \text{the decomposition \eqref{ito_process}
       holds for} \\
  &\quad  u^d \in L^2_\cP(\Omega;H^1(0,T;H)), \ u^s \in
    L^2_\cP(\Omega;L^2(0,T;\cL^2(U,V))) \}.
\end{align*}
 Note that the representation $u = u^d + u^s \cdot W$ is unique 
 and defines  an isomorphism
$$\cU \simeq L^2_\cP(\Omega;H^1(0,T;H))\times
L^2_\cP(\Omega;L^2(0,T;\cL^2(U,V))).$$
In the following, we will systematically (and tacitly) use such isomorphism by
identifying $u\in \cU$ with the corresponding pair of processes
$(u^d,u^s) \in L^2_\cP(\Omega;H^1(0,T;H))\times
L^2_\cP(\Omega;L^2(0,T;\cL^2(U,V))) $. We will also make use of the
the subspace $\cV \subset \cU$ given by
$$\cV:=\cU \cap C([0,T];L^2(\Omega,\cF;V)\}.$$

We ask 
$\phi:H\to[0, \infty]$ to be a convex and lower
  semicontinuous with $\phi(0)=0$  and {\it essential domain}   $
  D(\phi)=\{u \in H  \ : \ \phi(u) <\infty\}=V$,  and require the
  subdifferential 
  $\partial \phi:
  D(\partial\phi)\subseteq H \to H$, where $ D(\partial\phi) =\{u \in H  \ :
  \ \partial \phi(u) \not =\emptyset\}$, to be single-valued and 
  coercive in the following sense: 
  \begin{align}
    \exists\, c_\phi>0: \quad   (\partial \phi(u_1) -
    \partial \phi(u_2),u_1 - u_2)\geq c_\phi \| u_1 - u_2 \|^2_V  
    \quad\forall u_1, \, u _2 \in D(\partial \phi).\label{ass:phi}
  \end{align}
 %       Note that this implies that $\phi$ is coercive
%         on $V$ since
% \begin{equation}
%   \phi(u) = \int_0^1 ( \partial \phi(r u),
% u) \, \d r \geq  
% \frac{c_\phi}2  \| u\|_V^2\quad \forall u \in V.
% 
% \label{ass:phi3}
% \end{equation}
We moreover ask $\phi$ to be continuous at some point of  its
domain , so %with $ D(\partial \phi)\subset V$, so
 that the subdifferential $\partial(\phi|_V): V\to V^*$ of the
 restriction $\phi|_V:V \to [0, \infty]$ is maximal
 monotone and coincides with $\partial \phi$ on $V$. We further assume 
$\partial \phi :V\to V^*$ to be G\^ateaux-differentiable with
G\^ateaux-differential $
  D_{\mathcal G}\partial \phi\in C(V; \cL_s(V,V^*))$  fulfilling
 \begin{align}
  \label{ass:phi2}
\exists\, C_{\phi} >0:\qquad \|D_{\mathcal G}\partial \phi (u)\|_{\cL(V,V^*)}
   \leq  C_{\phi} \quad\forall u\in V.
 \end{align}
 Note that the latter  entails that $\partial \phi$ is linearly bounded
 from $V$ to $V^*$. Indeed, we have that 
 $$ \| \partial \phi(u)\|_{V^*} = \left\|\int_0^1 \langle D_{\mathcal G} \partial
 \phi(ru),u\rangle \, \d r \right\|_{V^*} \leq \int_0^1 \|D_{\mathcal G} \partial
 \phi(ru)
\|_{\cL(V,V^*)} \| u \|_V \, \d r \leq C_\phi \|u\|_V$$
and we can compute that
\begin{equation}
  \frac{c_\phi}{2}\| u \|_V^2 \leq  \phi(u) = \int_0^1 ( \partial \phi(r u),
u) \, \d r \leq \int_0^1 r  C_\phi\| u\|_V^2\, \d r =
\frac{C_\phi}{2}\| u \|_V^2\quad \forall u \in V.\label{eq:linearly}
\end{equation}
Moreover, we have  the control
 \begin{equation}
   \label{eq:control_L}
  \forall\, u \in \cU : \qquad {\rm Tr}_H\, L(u) \leq
   C_{\phi}  \|
  u^s \|_{\cL^2(U,V)}^2  \quad \text{a.e.~in~}\Omega\times(0,T), 
\end{equation}
where we recall that $L(u) := u^s (u^s)^* D_{\mathcal G} \partial \phi(u)$.

We require the map  $F:  [0,T]\times H \to H$ to be Carath\'eodory
with $F(\cdot,0)\in L^2 (0,T;H)$  and to 
be Lipschitz continuous, uniformly with respect to $t $. More precisely, we assume that 
\begin{align}
  &\exists\, c_F>0: \quad   \|F(t,u_1) - F(t,u_2)\| \leq c_F\|u_1 - u_2\|
  \nonumber\\
  &\quad  \forall u_1,\, u_2 \in H, \ \text{for a.e.}\   t \in   (0,T). \label{ass:F}
\end{align}
The latter specifically implies that the process $F(\cdot, u)$ belongs
to $L^2_\cP(\Omega;
L^2(0,T;H))$ for all $u \in L^2_\cP(\Omega;
L^2(0,T;H))$.

The map $G:  [0,T] \times V \to \cL^2(U,V)$ is also asked
to be Carath\'eodory with $G(\cdot,0)\in L^2 (0,T;\cL^2(U,V))$ 
and uniformly Lipschitz continuous and linearly bounded  with respect to $t$, namely,  
\begin{align}
  &\exists\, c_G>0: \qquad   \|G( t,u_1) - G(  t,u_2)\|_{\cL^2(U,H)} \leq c_G\|u_1 - u_2\| \nonumber\\
  &\quad \forall u_1,\, u_2 \in H,  \ \text{for a.e.}\  t \in (0,T),\label{ass:G}
\end{align}
and 
\begin{equation}
  \exists\, c_{G,2}>0: \qquad   \|G(\cdot,u)\|_{\cL^2(U,V)}\leq
    c_{G,2}\left(1+\|u \|_V\right)\quad \forall u \in V. \label{ass:G0}
  \end{equation}
In particular, $G(\cdot, u)\in L^2_\cP(\Omega;
L^2(0,T; \cL^2(U,V)))$ for all $u \in L^2_\cP(\Omega;
L^2(0,T;V))$.

Eventually, we prescribe the initial datum
\begin{equation}
  \label{eq:initial}
  u^0 \in L^2(\Omega,\cF_0;V).
\end{equation}

The Cauchy problem for \eqref{eq:0} can hence be specified as
\begin{equation}
  \label{eq:00}
  u(t) =  u^0 + \int_0^t (F(\cdot,u)-\partial \phi(u)) \, \d s + \int_0^t
  G(\cdot,u)\, \d W\quad \forall\, t \in [0,T], \ \P\text{-a.s.}
\end{equation}
where the latter is intended as an equation in $H$.

Under the above assumptions, one can adapt the theory from
\cite{biot}  and \cite{Gess-strong} 
 in order to prove that equation \eqref{eq:00} admits a
unique solution $u \in
\cV$ which in addition belongs to
\begin{align*}
  &L^2_{\cP}(\Omega;L^\infty(0,T;V)) := \big\{v:\Omega\to
    L^\infty(0,T; V)  \ \text{weakly*
      progressively measurable}\nonumber\\
&\hspace{48mm} \text{with} \ \ 
  \E\|v\|_{L^\infty(0,T;V)}^2<\infty
                                             \big\},
                                             \end{align*}
and   can be obtained  as limits of
approximations arising from  Yosida-regularizing
$\partial\phi$. As such, when referring to 
 a  solution of equation
\eqref{eq:00} the regularity $u \in \cV$ will be always
assumed in the following. Note that this is not restrictive, for all
strong-in-time solutions of \eqref{eq:00}, namely, 
$u^d\in L^2(\Omega; W^{1,1}(0,T;H))$, can a posteriori be proved to
 belong to 
 $\cV$, see Remark \ref{w11} below.

The central observation of this note is the following characterization.  

\begin{theorem}[Energy Dissipation Principle]\label{thm:main}
  $u\in \cU$ solves \eqref{eq:00} if and only if $0=I(u) = \min_{\cU} I$.
\end{theorem}

This characterization is proved in the next Section \ref{sec:proof},
by resorting to a generalized It\^o formula for $\phi$.   

%%%%%%%%%%%%%%%%%%%%%%%%%%%%%%%%%%%%%%%%%%%%%%%%
\section{It\^o formula and proof of Theorem \ref{thm:main}}
\label{sec:proof}

In the deterministic case, the   Energy-Dissipation Principle hinges on the validity of the
chain rule for the  functional $\phi$.  In the stochastic case, this
corresponds to a specific It\^o formula, which we now present.

\begin{proposition}[It\^o formula]
  \label{prop:ito}
  Let $u \in \cV$ and assume 
  that $\partial \phi(u)\in L^2_\cP(\Omega;
  L^2(0,T;H))$. Then,
  \begin{align}
  \phi(u(t)) &=  \phi(u(0)) +  \int_0^t(\partial_t u^d,\partial
    \phi(u))\,\d s  + \int_0^t( \partial \phi(u),u^s\, \d W)
    \nonumber\\
& + 
   \frac12\int_0^t\operatorname{Tr}_H L(u) \,\d s  \quad 
  \forall t\in[0,T], \ \P\text{-a.s.}\label{eq:ito0}
  \end{align}
In particular, $ t \mapsto \E\phi(u(t))\in W^{1,1}(0,T)$ and
  \begin{align}
  &\frac{\d}{\d t}\E\phi(u(t)) =   \E  (\partial_t u^d,\partial
    \phi(u)) +
   \frac12\E \operatorname{Tr}_H\, L(u)  \quad 
   \text{for a.e.} \ t\in(0,T).\label{eq:ito}
  \end{align}
\end{proposition}

\begin{proof} The It\^o formula
  \eqref{eq:ito0} is proved in \cite[Lemma 3.2]{biot} for $\phi$ replaced by its   Moreau-Yosida
  approximation $\phi_\lambda$ at level $\lambda>0$
  \cite{brezis73}.  In particular,  for all $\lambda>0$ we have that 
\begin{align}
  \phi_\lambda(u(t)) &=  \phi_\lambda(u(0)) +  \int_0^t(\partial_t u^d,\partial
    \phi_\lambda(u))\,\d s  + \int_0^t( \partial \phi_\lambda(u),u^s\, \d W)
    \nonumber\\
& + 
   \frac12\int_0^t\operatorname{Tr}_H (u^s (u^s )^* D_{\mathcal G}
         \partial \phi_\lambda (u))\,\d s  \quad 
         \forall t \in [0,T], \ \P\text{-a.s.}\label{eq:ito0l}
\end{align}  In order to check for
\eqref{eq:ito0}, we hence aim at taking the limit $\lambda \to 0$ in
\eqref{eq:ito0l}. The pointwise convergence
of $\phi_\lambda$ to $\phi$ on $D(\phi)$ \cite[Prop.~2.11, p.~39]{brezis73} ensures that $\phi_\lambda(u(t)) \to
\phi(u(t))$ and $\phi_\lambda(u(0)) \to \phi(u(0))$. Moreover, from $\| \partial \phi_\lambda (u)\| \leq \| \partial \phi  (u)\| $
a.e. and the fact that $\partial \phi(u)\in L^2_\cP(\Omega;
  L^2(0,T;H))$ we conclude that $\partial \phi_\lambda(u) \weakto
  \xi $ in $L^2_\cP(\Omega;
  L^2(0,T;H))$  by  possibly extracting a not relabeled subsequence. On the other hand, one readily check that $\xi=\partial \phi(u)$
  a.e. by passing to the limit  $\lambda \to 0$ into
$(\partial\phi_\lambda (u),w-u) \leq\phi_\lambda(w) - \phi_\lambda(u)$
a.e. 
for all $w\in L^2_\cP(\Omega;
  L^2(0,T;H))$. Hence, extracting a subsequence  was actually not
  needed. Eventually, the first two integrands in the right-hand side of
  \eqref{eq:ito0} converge to the corresponding limits. 

  We are hence left to check the limit of the trace term. To this aim,
  we recall from \cite[Lemma 3.1]{biot} that
  $$D_{\mathcal G}
         \partial \phi_\lambda (u) = D_{\mathcal G}
         \partial \phi(J_\lambda (u)) \, D_{\mathcal G} J_\lambda (u)$$
         where we have denoted by $J_\lambda : V^* \to V$  the {\it resolvent} $J_\lambda :=
         (I + \lambda \partial \phi)^{-1}$. The  coercivity assumption
   \eqref{ass:phi} ensures that $J_\lambda$ is
         Lipschitz continuous. 
% One can compute
%          \begin{align*}
%            &{\rm Tr}_H \, (u^s (u^s)^* D_{\mathcal G} \partial
%            \phi_\lambda (u)) = {\rm Tr}_H \, (u^s (u^s)^* D_{\mathcal G}
%          \partial \phi(J_\lambda (u)) \, D_{\mathcal G} J_\lambda (u))
          %  \\
         %   &\quad = \sum_{k=1}^\infty \big( (u^s (u^s)^* D_{\mathcal G}
         % \partial \phi(J_\lambda (u)) \, D_{\mathcal G} J_\lambda (u))
         %     e_k, e_k \big)  \\
         %   &\quad= \sum_{k=1}^\infty \big(  D_{\mathcal G} J_\lambda (u)
         %     e_k, (D_{\mathcal G}
         % \partial \phi(J_\lambda (u)))^* u^s (u^s)^*  e_k \big).
        % \end{align*}
 Moreover, one can enhance the usual a.e. convergence $J_\lambda u \to u$ in $H$
   to
   \begin{align*}
     &c_\phi \| J_\lambda u - u \|^2_{L^2_\cP(\Omega;L^2(0,T;V))} \leq
   \E \int_0^T(\partial \phi(J_\lambda u) - \partial \phi(u),
       J_\lambda u - u) \, \d s  \to 0,
     \end{align*}
so that  $J_\lambda u \to u$ in $V$
   a.e.  Recalling that 
$D_{\mathcal G}
         \partial \phi \in C(V;\cL_s(V,V^*))$ one gets that
         $$ u^s (u^s)^*D_{\mathcal G}
         \partial \phi(J_\lambda (u))  \to u^s (u^s)^*D_{\mathcal G}
         \partial \phi(u)   \quad \text{in} \
         \cL^1_s(V,V)  \ \text{a.e.}$$
         On the other hand, from  \cite[Lemma 3.1]{biot}
         one has that $D_{\mathcal G} J_\lambda(h) \to I$ in
         $\cL_s(H,H)$ for all $h\in H$. In fact, under the coercivity assumption \eqref{ass:phi}, the argument of  \cite[Lemma 3.1]{biot} can be
        straightforwardly extended to ensure that the convergence
        $D_{\mathcal G} J_\lambda(v) \to I$ actually holds in 
        $\cL_w(V,V)$ for all $v \in V$, as well. In particular, for
        all $k\in \Nz$  we have that
        \begin{align*}
          &\big(u^s (u^s)^* D_{\mathcal G}
         \partial \phi(J_\lambda (u)) \, D_{\mathcal G} J_\lambda (u)
          e_k, e_k\big) = \big(D_{\mathcal G} J_\lambda (u) e_k, D_{\mathcal G}
            \partial \phi(J_\lambda (u))   u^s  (u^s)^* e_k \big) \\
          &\quad \to
          \big(  e_k, D_{\mathcal G}
         \partial \phi(u)   u^s  (u^s)^* e_k \big) = (L(u)e_k,e_k).
        \end{align*}
 In order to use the latter and pass to the limit in 
        \begin{equation}{\rm Tr}_H \,  (u^s (u^s)^* D_{\mathcal G}
         \partial \phi(J_\lambda (u)) \, D_{\mathcal G} J_\lambda
         (u)) = \sum_{k=1}^\infty
\big(u^s (u^s)^* D_{\mathcal G}
         \partial \phi(J_\lambda (u)) \, D_{\mathcal G} J_\lambda (u)
          e_k, e_k\big) 
          \label{serie}
          \end{equation}
we now provide a bound on the series, independently of $\lambda$. We
recall the invariance of the trace
        under permutations, namely, $$  {\rm Tr}_H \,  (u^s (u^s)^* D_{\mathcal G}
         \partial \phi(J_\lambda (u)) \, D_{\mathcal G} J_\lambda (u))
         =   {\rm Tr}_U\,( (u^s)^* D_{\mathcal G} 
         \partial \phi(J_\lambda (u)) \, D_{\mathcal G} J_\lambda (u)
         u^s),$$ where now ${\rm Tr}_U\,$ is the trace in $\cL^1(U,U)$, related to a
         given (hence any) orthonormal basis $(v_k)_{k\in \Nz}\subset U$. In particular, we
         have that 
         \begin{align*}
{\rm Tr}_H \,  (u^s (u^s)^* D_{\mathcal G}
         \partial \phi(J_\lambda (u)) \, D_{\mathcal G} J_\lambda (u))
           =   \sum_{i=1}^\infty\big(  D_{\mathcal G} 
         \partial \phi(J_\lambda (u)) \, D_{\mathcal G} J_\lambda (u)
           u^sv_k, u^sv_k \big)
         \end{align*}
       and we can argue as follows
       
         \begin{align*}
           &\left|\big(  D_{\mathcal G} 
         \partial \phi(J_\lambda (u)) \, D_{\mathcal G} J_\lambda (u)
             u^sv_k, u^sv_k \big)\right| \\
           &\quad\leq  \| D_{\mathcal G} 
         \partial \phi(J_\lambda (u)) \|_{\cL(V,V^*)} \| D_{\mathcal
             G} J_\lambda (u) \|_{\cL(V,V)}\| u^s v_k\|^2_V \\
           &\quad\leq
  \frac{C_{\phi}}{c_\phi}
            \| u^s v_k\|^2_V   \in\ell^1. 
         \end{align*} 
         By the Dominated Convergence Theorem we have hence proved that 
         
         \begin{align*}
             {\rm Tr}_H \, (u^s (u^s)^* D_{\mathcal G}
         \partial \phi(J_\lambda (u)) \, D_{\mathcal G} J_\lambda (u))
           \to    {\rm Tr}_H \, L(u)   \quad \text{a.e.} 
         \end{align*}
        as well as  
         \begin{align*}
           &| {\rm Tr}_H \, (u^s (u^s)^* D_{\mathcal G}
         \partial \phi(J_\lambda (u)) \, D_{\mathcal G} J_\lambda
             (u))|  \leq \frac{C_{\phi}}{c_\phi}  \|
           u^s\|^2_{\cL^2(U,V)} \quad \text{a.e.}
\end{align*}
      As $\|
           u^s\|^2_{\cL^2(U,V)} \in L^1_\cP(\Omega; L^1(0,T))$ one can
           use again the Dominated Convergence Theorem, pass the limit in \eqref{eq:ito0l} as
         $\lambda \to 0$, and get \eqref{eq:ito0}.

         We now localize formula \eqref{eq:ito0} to a subinterval
         $[s,t] \subset (0,T)$ in order to get that
$$ 
  \E\phi(u(t))  = \E \phi(u(s)) +  \E\int_s^t(\partial_t u^d,\partial
    \phi(u))\,\d s   +
   \frac12\E\int_s^t\operatorname{Tr}_H L(u) \,\d s  \quad 
  \forall\,0 <  s < t  <  T$$
This proves that    $t\mapsto \E\phi (u
         (t))$ is absolutely continuous, and the differential It\^o
formula~\eqref{eq:ito} follows  by the arbitrariness of $s$ and $t$.
\end{proof}
 
We now use Proposition \ref{prop:ito} in order to give an equivalent
formulation of the EDP functional $I$ in terms of squared residuals.

\begin{proposition}[Equivalent formulation] \label{prop:equiv}
  For $u \in \cV$ with $\partial \phi(u)\in
  L^2_{\cP}(\Omega;L^2(0,T;H))$   one has 
  \begin{align}
    &I(u) =  \frac12 \E \int_0^T \| \partial_t u^d +\partial \phi(u) -
    F(\cdot,u)\|^2 \, \ds+ 2C_\phi \E \int_0^T \| u^s -
    G(\cdot,u)\|^2_{\cL^2(U,V)} \, \ds\nonumber\\
    &\qquad + \E\| u(0) - u^0\|^2_V  . \label{eq:equiv}
  \end{align} 
\end{proposition}
\begin{proof}
  Under the assumptions  $u \in \cV$ and $\partial \phi(u)\in
  L^2_{\cP}(\Omega;L^2(0,T;H))$  the It\^o formula
  \eqref{eq:ito0} holds and we can compute
   \begin{align*}
  I(u) &= \E\phi(u(T)) - \E \phi(u(0)) + \frac12 \E \int_0^T \|
  \partial_t u^d\|^2 \, \ds
+ \frac12 \E \int_0^T \|
   F(\cdot,u) -  \partial
    \phi (u) \|^2 \, \ds\nonumber \\
  &\qquad  - \frac12 \E \int_0^T {\rm Tr}_H\,L(u)  \, \ds- \E \int_0^T
  (\partial_t u^d, F(\cdot,u)) \, \ds\nonumber \\
  &\qquad+
    2C_\phi\E\int_0^T \| u^s  - G(\cdot,u)\|^2_{\cL^2(U,H)}  \, \ds
    +\E \| u(0)-u^0\|^2_V  %+ |\E \phi(u(0)) - \E\phi(u^0)|
    \nonumber \\
  &\stackrel{\eqref{eq:ito0}}{=} \frac12 \E \int_0^T \|
  \partial_t u^d \|^2 \, \ds+ \frac12 \E \int_0^T \|
   F(\cdot,u) -  \partial
    \phi(u) \|^2 \, \ds\nonumber \\
   &\qquad + \E \int_0^T (\partial_t u^d, \partial
    \phi(u)-F(\cdot,u) ) \, \ds+
   2C_\phi\E\int_0^T \| u^s  - G(\cdot,u)\|^2_{\cL^2(U,V)}  \, \ds\nonumber \\
       &\qquad +\E \| u(0)-u^0\|^2_V %+ |\E \phi(u(0)) - \E\phi(u^0)|
         \nonumber \\
  &= \frac12 \E \int_0^T \| \partial_t u^d + \partial
    \phi(u)-F(\cdot,u)\|^2 \, \ds
    +2C_\phi\E \int_0^T \| u^s - G(\cdot,u)\|^2_{\cL^2(U,V)} \, \ds\nonumber \\
       &\qquad +\E \| u(0)-u^0\|^2_V.% + |\E \phi(u(0)) - \E\phi(u^0)|.
         \qedhere
\end{align*}
\end{proof}

Owing to the equivalence from Proposition \ref{prop:equiv} we are now
in the position of checking the characterization Theorem \ref{thm:main}.

\begin{proof}[Proof of Theorem \ref{thm:main}]
Let $u \in \cU$ be such that $I(u)=0$. The boundedness
of $I$ in particular entails that $u \in \cV$ and that  the difference   $ \partial\phi(u)-F(\cdot,u)$ belongs to
$L^2_\cP(\Omega;L^2(0,T;H))$. As $u\in  
L^2_\cP(\Omega;L^2(0,T;H))$ one has  $F(\cdot,u)\in
L^2_\cP(\Omega;L^2(0,T;H))$ owing to \eqref{ass:F}. This implies that $\partial \phi(u)$ is in
$L^2_\cP(\Omega;L^2(0,T;H))$. We can hence  use equation  \eqref{eq:equiv} and obtain that
  \begin{align*}
    & \frac12 \E \int_0^T \| \partial_t u^d +\partial \phi(u) -
    F(\cdot,u)\|^2 \,\ds+ 2C_\phi \E \int_0^T \| u^s -
    G(\cdot,u)\|^2_{\cL^2(U,V)}\,\ds \\
    &\quad + \E\| u(0) - u^0\|^2_V %+ |\E \phi(u(0)) - \E\phi(u^0)|
      \stackrel{\eqref{eq:equiv}}{=} I(u)=0.
  \end{align*}
  This proves that $\partial_t u^d + \partial \phi(u) = F(\cdot,u)$
  and $u^s=G(\cdot,u)$ a.e. in $\Omega\times (0,T)$ and $u(0) =u^0$ $\P$-a.s. Hence, $u$ solves
  equation \eqref{eq:00}.

  Let now $u\in \cV$ solve equation \eqref{eq:00}. In particular, we
  have that $\partial_t u^d + \partial \phi(u) = F(\cdot,u)$
  and $u^s=G(\cdot,u)$ a.e. in $\Omega\times (0,T)$ and $u(0) =u^0$ $\P$-a.s. As $\partial_t u^d
 \in  L^2_\cP (\Omega;L^2(0,T;H))$ and $F(\cdot, u)  \in  L^2_\cP
 (\Omega;L^2(0,T;H))$ from \eqref{ass:F}, we have that 
 $\partial \phi
  ( u)  \in  L^2_\cP (\Omega;L^2(0,T;H))$ as well. Again, the
  equivalence \eqref{eq:equiv} holds and we have that 
\begin{align*}
  I(u) 
  &\stackrel{\eqref{eq:equiv}}{=}  
   \frac12 E \int_0^T \| \partial_t u^d + \partial
    \phi(u)-F(\cdot,u)\|^2 \, \ds 
    +2C_\phi\E \int_0^T \| u^s - G(\cdot,u)\|^2_{\cL^2(U,V)}\, \ds  \nonumber \\
  &\quad +\E \| u(0)-u^0\|^2_V=0 .%+ |\E \phi(u(0)) - \E\phi(u^0)|=0.
\end{align*}
This concludes the proof of the theorem. 
\end{proof}

\begin{remark}\label{w10}\rm
  The equivalence \eqref{eq:equiv} reveals that the EDP functional is indeed
  nothing but the square residual of the system
$$\partial_t u^d + \partial \phi(u) = F(\cdot, u), \quad u^s =
G(\cdot,u^s),\quad  u(0)=u^0. $$
In particular, the expression in
\eqref{eq:equiv} could have been used as alternative and possibly more
informative definition for $I$. On the other hand, definition
\eqref{eq:EDP} is the direct stochastic extension of the classical one
\cite{Ambrosio,Dondl19} and has the advantage of making sense also
out of the purely Hilbertian setting. Without going into the greatest generality,
which would call for considering stochastic integrals in Banach
spaces, let us mention that the present results (in particular, the
validity of Theorem \ref{thm:main}) could be extended to the EDP
functional
\begin{align}
 I (u) &= 
          \E\phi(u(T)) - \E \phi(u(0)) +    \E \int_0^T\psi_A(\partial_t u^d,
    -\partial\phi(u)+F(\cdot,u) ) \, \ds  \nonumber \\
  &\quad - \E \int_0^T
  (\partial_t u^d, F(\cdot,u))\, \ds  -  \E \int_0^T {\rm
    Tr}_H\, L(u)\, \ds 
   \nonumber \\
  &\quad  + 2C_\phi\E\int_0^T \| u^s -
    G(\cdot,u)\|^2_{\cL^2(U,V)} \, \ds +\E\|u(0)- u^0\|^2_V .% \nonumber \\[2mm]&\quad+ |\E \phi(u(0)) - \E \phi(u^0)|.
\label{eq:EDP2}
\end{align}
 Here, $\psi_A: H \times H \to (-\infty,\infty]$ is a convex
function {\it representing} the maximal monotone operator $A:H \to H$
in the sense of the Fitzpatrick theory \cite{Fitzpatrick}, see also
\cite{Visintin1,Visintin2} for
additional material and details. In particular,
\begin{align}
  & \psi_A(v,w) \geq (v,w) \quad \forall v,\, w \in H, \label{eq:fenchel1}\\
   & \psi_A(v,w) = (v,w) \quad  \Leftrightarrow \quad w \in A(v).\label{eq:fenchel2}
\end{align}
An example for such $\psi_A$ is the so-called {\it Fitzpatrick
  function}
$$\psi_A(v,w)  = \sup \{ (\hat v,w) + (v,\hat w) - (\hat v, \hat w)\,
: \, \hat v,\,\hat w \in H, \, \hat w \in A(\hat v)\}.$$
If $A$ is cyclic, namely $A= \partial \eta$ for some $\eta: H \to
(-\infty,\infty]$ convex, proper, and lower semicontinuous, a second
example for $\psi_A$ is the  {\it Fenchel
  function}
$$\psi_A(v,w) = \eta(v) + \eta^*(w)$$
where $\eta^*$ is the Legendre conjugate of $\eta$.

By using \eqref{eq:fenchel1}-\eqref{eq:fenchel2} one can prove that
null-minimizers of $I$ from \eqref{eq:EDP2} solve the 
{\it doubly nonlinear}
equation
$$A(\partial_t u^d)\, \d t + u^s \, \d W + \partial\phi
(u)\, \d t  = F(\cdot, u)\, \d t  + G(\cdot, u) \, \d W$$
where $A: H \to H$ is maximal monotone, nondegenerate, and linearly
bounded but not necessarily cyclic.
The latter,  under suitable assumptions, has been proved to admit martingale solutions in~\cite{biot}.

\end{remark}

\begin{remark}\label{w11}\rm
 By adapting the argument of Proposition \ref{prop:ito} one can check
 that strong-in-time solutions with $u^d\in L^2(\Omega; W^{1,1}(0,T;H))$
 of equation \eqref{eq:00} are actually in $\cV$, so that assuming $u\in \cV$
 is actually not restrictive. Indeed, given $u^d\in L^2(\Omega; W^{1,1}(0,T;H))$
 and taking \eqref{ass:F}--\eqref{ass:G} into account 
 one has that $\partial \phi(u)\in L^2_\cP(\Omega; L^1(0,T;H))$,
 $F(\cdot, u)\in  L^2_\cP(\Omega; L^2(0,T;H))$, and $G(\cdot,u) \in
 L^2_\cP(\Omega; L^2(0,T;\cL^2(U,V)))$. In order to conclude for $u\in
 \cU$ it hence suffices to prove that indeed $\partial \phi(u)\in
 L^2_\cP(\Omega; L^2(0,T;H))$. At some
 approximation level (for instance, that of Yosida approximations $u_\lambda$ from
 the proof of Proposition \ref{prop:ito}), the It\^o
 formula holds for $\partial \phi(u)$ in $L^2_\cP(\Omega;
 L^1(0,T;H))$, as well. On solutions of the equation \eqref{eq:00}
 one can hence replace $\partial_t u^d$ in the It\^o formula
 \eqref{eq:ito} by $F(\cdot,u)-\partial\phi(u)$ and easily check that
 indeed $\partial \phi(u)\in
 L^2_\cP(\Omega; L^2(0,T;H))$.  In fact,  by a standard application
 of the Burkholder-Davis-Gundy inequality one additionally obtains
 that solutions of  \eqref{eq:00}  belong to 
 $L^2_\cP(\Omega;L^\infty(0,T;V))$, as well.  
\end{remark}

%%%%%%%%%%%%%%%%%%%%%%%%%%%%%%%%%%%%%%%%%%%%%%%%
\section{ Some properties of the EDP functional}
\label{sec:coercivity}

As mentioned above, under the assumptions of Section \ref{sec:setting}
equation \eqref{eq:00} admits a unique solution $u$. In particular, the
null-minimization problem for $I$ is uniquely solvable.

 In this section, we comment on the possibility of tackling the
null-minimization problem for $I$ directly. We prove that $I$ is coercive
in $\cV$ (Proposition \ref{prop:coercivity}) and that minimizers are actually
null-minimizers (Proposition \ref{prop:null}). Moreover, we check that
$I$ is lower semicontinuous, up to possibly changing the underlying
stochastic basis (Proposition \ref{prop:lsc}). Unfortunately, this
lower semicontinuity property is 
too weak to allow for an application of the Direct Method, preventing
us from obtaining a complete alternative existence proof for
\eqref{eq:00}.

The case of an atomic $\Omega$ is special. Here, no change in
the stochastic basis is needed for lower semicontinuity and the
null-minimization of $I$ can be directly carried out, bringing to a
fully variational existence proof for \eqref{eq:00}.

\begin{proposition}[Coercivity]\label{prop:coercivity}
  The sublevels of $I$ are bounded in $\cV$.
\end{proposition}

\begin{proof} Assume  $I(v) <\infty$. Then $v \in
  \cV \subset L^2_\cP(\Omega;L^2(0,T;H))$ and we have
   $F(\cdot,v) \in  L^2_\cP(\Omega;L^2(0,T;H))$ and
  $G(\cdot, v) \in  L^2_\cP(\Omega;L^2(0,T;\cL^2(U;V)))$ from \eqref{ass:F}-\eqref{ass:G}. As
  $I(v)$ is finite, we deduce that    $\partial\phi(v) \in
  L^2_\cP(\Omega;L^2(0,T;H))$ as well. %  One can hence apply Proposition
  % \ref{prop:equiv} and find that $\partial_t v^d + \partial
  %     \phi(v) - F(\cdot,v) $ is bounded   in
  %     $L^2_\cP(\Omega;L^2(0,T;H))$, $u^s - G(\cdot,u)$ is bounded   in
  %     $L^2_\cP(\Omega;L^2(0,T;\cL^2(U,V)))$,  
  % $v(0)$ is bounded in $L^p(\Omega,\cF_0;V)$, and $\E\phi(v(0))$ is
  % bounded in $\Rz$ in terms of $I(v)$.
  Applying the It\^o formula
      \eqref{eq:ito} and integrating
      on the interval $[0,t]$ for $t \in [0,T]$ we deduce that
      \begin{align}
         &\E \phi(v(t)) - \E \phi(v(0)) + \frac12 \E \int_0^t \|
        \partial_t v^d\|^2\, \d s  +\frac12 \E \int_0^t \|
           \partial\phi(v) - F(\cdot,v)\|^2 \, \d s\nonumber\\
        &\quad - \E
        \int_0^t(F(\cdot, v), \partial_t v^d)\, \d s - \frac12 \E
        \int_0^t {\rm Tr}_H\, L(v) \, \d s+ 2C_\phi\E
      \int_0^T\|v^s  - G (\cdot,v)\|^2_{\cL^2(U,V)}\,
                     \d s \nonumber\\
        &\quad = \frac12 \E \int_0^t \| \partial_t v^d +
        \partial\phi(v) - F(\cdot,v)\|^2\, \d s+ 2C_\phi\E
      \int_0^t\|v^s  - G (\cdot,v)\|^2_{\cL^2(U,V)}\,
                     \d s  \nonumber\\
        &\quad \leq I(v)\label{eq:coerc}.
        \end{align}  
We now use the coercivity  \eqref{eq:linearly}  in order to get that
 \begin{align}
         &\frac{c_\phi}2
          \E\| v(t)\|_V^p + \frac14 \E \int_0^t \|
        \partial_t v^d\|^2 \, \d s +\frac12 \E \int_0^t \|
           \partial\phi(v) - F(\cdot,v)\|^2  \, \d s \nonumber\\
   &\qquad+ \frac{3C_\phi}{2}\E
      \int_0^T\|v^s\|^2_{\cL^2(U,V)}\,
                     \d s \nonumber\\
   &\quad
    \leq   \E \phi(v(0)) + \E
        \int_0^t\|F(\cdot, v)\|^2 \, \d s+ \frac12 \E
        \int_0^t {\rm Tr}_H\, L(v)   \, \d s \nonumber\\
   &\qquad+8C_\phi\E
        \int_0^t\|G(\cdot, v)\|^2_{\cL^2(U,V)} \, \d s+ I(v). \nonumber
        \end{align}
By the Lipschitz continuity of $F$ and
 the linear boundedness of $G$ from \eqref{ass:F}--\eqref{ass:G0} 
      and the  bounds \eqref{eq:linearly}-\eqref{eq:control_L}
       we get
      \begin{align}
         & \frac{c_\phi}{2}  \E\| v(t)\|_V^2 + \frac14 \E \int_0^t \|
        \partial_t v^d\|^2  \,\d s +\frac12 \E \int_0^t \|
           \partial\phi(v) - F(\cdot,v)\|^2\,\d s  \nonumber\\
   &\qquad+  \frac{3C_\phi}{2}\E
      \int_0^T\|v^s\|^2_{\cL^2(U,V)}\,
                     \d s
    \nonumber\\
   &\quad\leq    \frac{C_\phi}{2}\E \|v(0)\|_V^2  + C         
        \int_0^t\E\|v\|_V^2\,\d s
         + \frac{C_\phi}{2}
        \int_0^t\E \| v^s\|_{\cL^2(U,V)}^2\,\d s + C+ I(v). \nonumber
      \end{align}
      for some positive constant $C$, depending on   the data
      $c_\phi$,  $C_\phi$, $c_F$,
      $c_G$, 
      $c_{G,2}$,  $ \|u^0\|_V$,  and 
        $\| F(\cdot, 0)\|_{L^2(0,T;H)}$, but
      independent of $v$.  An application of the Gronwall Lemma ensures that
$$ \max_{[0,T]}\, \E \| v \|^2_V + \E \int_0^T \|
\partial_t v^d\|^2\, \d s+ \E
      \int_0^T\|v^s\|^2_{\cL^2(U,V)}\,
                     \d s \leq C (1 + I(v)),$$
possibly by updating the constant. The assertion follows.
\end{proof}

In order to discuss lower limits, we   make the notation for the EDP functional more precise  by explicitly indicating the background
stochastic structure  and the given initial value.  When needed in the following, we use
the extended notation $$u  \mapsto  \hat I (  \Omega,
\cF,(\cF_t)_{t\in[0,T]} ,  \P,  W, u^0 ,   u)$$ instead of  $u  \mapsto
I(u)$. Correspondingly, we specify the dependence on the stochastic
basis of
the space of It\^o processes by using the notation $\hat\cU(  \Omega,
\cF,(\cF_t)_{t\in[0,T]} ,  \P,  W)$.
Our lower-semicontinuity result reads as follows. 
 
      \begin{proposition}[$\liminf$  tool]\label{prop:lsc}
     %   For any sequence $(u_\eps)_{\eps>0} \in \cU $  
       %  such that $\sup_\eps I(u_\eps)<\infty$, 
     For all $u_\eps \wstarto u$ in $\cV$ one can find  a stochastic basis $(\hat  \Omega, \hat
        \cF,(\hat  \cF_t)_{t\in[0,T]} , \hat   \P)$, a  not
        relabeled  sequence of
        measurable maps $\eta_\eps: (\hat \Omega, \hat \cF) \to
        (\Omega, \cF)$ with $\P \circ \eta_\eps = \hat \P$, a cylindrical Wiener
        process $\hat W$ on $U$, a process $$\hat u \in \hat \cU (\hat  \Omega, \hat
        \cF,(\hat  \cF_t)_{t\in[0,T]} , \hat   \P, \hat W) \cap
        C([0,T];L^2(\hat \Omega,\hat \cF;V)),$$  and an initial
        value $\hat u^0 \in L^2(\hat \Omega, \hat \cF_0;V)$  such that
        $u_\epsi \circ \eta_\epsi \to \hat u$ in $C([0,T];H)$ a.e.~in $\hat \Omega$ and
        \begin{equation}
          \hat I(\hat  \Omega, \hat
        \cF,(\hat  \cF_t)_{t\in[0,T]} , \hat   \P , \hat W,  \hat u^0,\hat
        u)\leq \liminf_{\eps\to 0}
        I(u_\eps).\label{eq:liminf}
        \end{equation}
\end{proposition}

\begin{proof}
%  If $\liminf_\eps
%  I(u_\eps)=\infty$ the assertion follows by choosing $(\hat  \Omega, \hat
%        \cF,(\hat  \cF_t)_{t\in[0,T]} , \hat   \P)=(  \Omega, 
%        \cF,(  \cF_t)_{t\in[0,T]} ,    \P)$, $\eta_\eps={\rm id}$,
%        $\hat W = W$,
%        and $\hat u =u^0$.
%        
% With no loss of generality we can hence assume
%  that $\sup_\eps I_\eps(u_\eps) <\infty$, 
   As %$\sup_\eps I_\eps (u_\eps)<\infty$ we have that 
   $(u_\eps)_{\eps>0}$ is
   bounded in $\cV$, %by Proposition
%  \ref{prop:coercivity}. 
  %Note that 
the classical result \cite[Lemma
  2.1]{fland-gat} ensures that
  $$\iota_\epsi:=\int_0^\cdot u^s_\eps \, \d W$$
  is uniformly bounded in 
   $L^2_\cP(\Omega;H^\mu(0,T;V))$  for some
  $\mu \in (0,1/2)$.
Since $V$ is compact in $H$, the Aubin-Lions  Lemma \cite{Simon}
 ensures that
  
  \begin{align*}
   &H^\mu(0,T; V)\subset\subset L^2(0,T; H), \\
   &H^1(0,T;H) \subset \subset C([0,T];V^*),\\
   & L^2(0,T;V)  \cap\left(H^1(0,T;H) + H^\mu(0,T;V)\right)\subset\subset L^2(0,T; H).
  \end{align*} 
  This
  entails that the laws of $(u_\eps,u^0, u_\eps^d, \iota_\eps,W)$ are tight in
  \[
  L^2(0,T; H)\times V\times C([0,T]; V^*)\times L^2(0,T; H) \times C([0,T];U_1),
   \] 
   where $U_1$ is a separable Hilbert space 
   such that the inclusion $U\embed U_1$ is Hilbert-Schmidt.
    By the Skorohod Theorem \cite[Thm. 2.7]{ike-wata} one can hence
  find another probability space $(\hat \Omega, \hat \cF, \hat \P)$, a
  sequence of measurable maps $\eta_\eps: (\hat \Omega, \hat \cF) \to
   (\Omega, \cF)$ with $\P \circ \eta_\epsi = \hat \P$  for all $\eps>0$, 
   and some measurable
   
   \[
   (\hat u, \hat u^0,\hat u^d, \hat \iota, \hat W): (\hat \Omega, \hat \cF) \to
   L^2(0,T; H)\times V\times C([0,T]; V^*)\times L^2(0,T; H) \times C([0,T];U_1),
   \]
   
   such that, letting 
   $\hat u_\eps := u_\eps \circ \eta_\eps$,
    $\hat u_\eps^d := u_\eps^d \circ \eta_\eps$, 
    $\hat \iota_\eps := \iota_\eps \circ \eta_\eps$, and
  $\hat W_\eps := W_\eps \circ \eta_\eps$,
  
  \begin{align} 
 \hat u_\eps \to \hat u \quad& \text{in} \  L^2(0,T;H), \
                              \hat\P\text{-a.s.},  \label{conv:1}\\
 \hat u_\eps(0) \to \hat u^0\quad& \text{in} \  V, \
                              \hat\P\text{-a.s.},  \label{conv:1_bis}\\
 \hat u_\eps^d \to \hat u^d \quad& \text{in} \  C([0,T];V^*), \
                              \hat\P\text{-a.s.},  \label{conv:1}\\
    \hat \iota_\eps \to \hat \iota \quad& \text{in} \ L^2(0,T;H), \
                              \hat\P\text{-a.s.},\label{conv:11}\\
 \hat W_\eps \to \hat W \quad& \text{in} \ C([0,T];  U_1 ), \
                              \hat\P\text{-a.s.} \label{conv:2}
  \end{align}
  
  In fact, as $\eta_\eps$ preserves the laws,
  we also have,  setting $\hat u_\eps^s:=u_\eps^s\circ\eta_\eps$,  that 
   \begin{align} 
\partial_t\hat u^d_\eps \weakto \partial_t \hat u^d \quad& \text{in} \  
 L^2(\hat\Omega;L^2(0,T;H)) ,
    \label{conv:3} \\
     \hat u^s_\eps \weakto \hat u^s \quad& \text{in} \  
      L^2 (\hat\Omega;L^2(0,T;\cL^2(U,V)),  \label{conv:4}\\
 \hat u_\eps(0) \weakto \hat u^0 \quad& \text{in} \   L^2(\hat
                                         \Omega,\hat \cF_0;V), \label{conv:5}\\
     \partial \phi(\hat u_\eps) \weakto \hat \xi \quad& \text{in} \
                          L^2 (\hat\Omega;L^2(0,T;H)),  \label{conv:6}
   \end{align}
   The combination of convergences \eqref{conv:1} and \eqref{conv:6}
   entail that $\hat  \xi = \partial \phi(\hat u)$ a.e.  
Moreover, the
   Lipschitz continuity of $F$ and $G$ gives
   \begin{align}
     F(\cdot, \hat u_\eps)\to  F(\cdot, \hat u) \quad& \text{in} \  
      L^2 (\hat\Omega;L^2(0,T;H)) ,
                                                           \label{conv:7} \\
      G(\cdot, \hat u_\eps)\to G(\cdot, \hat u) \quad& \text{in} \ 
        L^2 (\hat\Omega;L^2(0,T;\cL^2(U,H)), 
    \label{conv:8} \\
     G(\cdot, \hat u_\eps)\weakto G(\cdot, \hat u) \quad& \text{in} \ 
       L^2 (\hat\Omega;L^2(0,T;\cL^2(U,V)). 
   \end{align}
    Setting $(\hat\cF_{\eps,t})_{t\in[0,T]}$ as the filtration 
   generated by $(\hat u_\eps,\hat u_\eps^d, \hat \iota_\eps,\hat W_\eps)$, 
   using again the fact that $\eta_\eps$ preserves laws, one has that
   \begin{equation}\label{prev}
     I(u_\eps) = \hat I (\hat  \Omega,
\hat  \cF,(\hat  \cF_{\eps,t})_{t\in[0,T]} ,  \hat  \P,  \hat
W_\eps ,   \hat u^0_\eps,  \hat
u_\eps).
\end{equation}
 Moreover, setting $(\hat\cF_{t})_{t\in[0,T]}$ as the filtration 
   generated by $(\hat u,\hat u^d, \hat \iota,\hat W)$, 
   a classical argument (see \cite{ScarStef-SDNL, biot}) ensures that 
   $\hat W$ is a $U$-cylindrical Wiener process, $\hat \iota=\hat u^s\cdot\hat W$, and 
   \[
   \hat u = \hat u^0 + \int_0^\cdot\partial_t\hat u^d(s)\,\d s +
   \int_0^\cdot \hat u^s(s)\,\d  \hat W  (s).
  \]
 Now, since we have that $\hat u \in  \hat \cU (\hat  \Omega, \hat
        \cF,(\hat  \cF_t)_{t\in[0,T]} , \hat   \P)  \cap C([0,T];
        L^2(\hat \Omega, \hat \cF;V))$,
        $\partial \phi(\hat u) \in L^2_{\hat \cP}(\hat
        \Omega;L^2(0,T;H))$,
        as well as
      $\hat u_\eps \in  \hat \cU (\hat  \Omega, \hat
        \cF,(\hat  \cF_{\eps,t})_{t\in[0,T]} , \hat   \P)  \cap C([0,T];
        L^2(\hat \Omega, \hat \cF;V))$, 
        $\partial \phi(\hat u_\eps) \in L^2_{\hat \cP_\eps}(\hat
        \Omega;L^2(0,T;H))$,  we can apply the equivalence
        \eqref{eq:equiv} and pass to the $\liminf$ owing to
        convergences \eqref{conv:3}-\eqref{conv:8} getting
\begin{align*}
  &\hat I (\hat  \Omega,
\hat  \cF,(\hat  \cF_t)_{t\in[0,T]} ,  \hat  \P,  \hat   W,   \hat u^0,  \hat u)
  \\
  &\quad \stackrel{\eqref{eq:equiv}}{=}
\frac12 \hat \E \int_0^T \| \partial_t \hat u^d
  +\partial \phi(\hat u) -
    F(\cdot,\hat u)\|^2 \, \d s+ 2C_\phi \hat \E \int_0^T \| \hat u^s -
    G(\cdot,\hat u)\|^2_{\cL^2(U,V)} \, \d s  \\
  &\qquad + \hat \E\| \hat u(0) -  \hat u^0 \|^2_V %+ |\hat \E \phi(\hat
  % u(0)) - \hat \E\phi(u^0)|
    \nonumber\\
  & \hspace{5mm} \leq \liminf_{\eps\to 0} \Bigg(
\frac12 \hat \E \int_0^T \| \partial_t \hat u^d_\eps
  +\partial \phi(\hat u_\eps) -
    F(\cdot,\hat u_\eps)\|^2 \, \d s  \\
  &\qquad + 2C_\phi \hat \E \int_0^T \| \hat u^s_\eps -
    G(\cdot,\hat u_\eps)\|^2_{\cL^2(U,V)} \, \d s + \hat \E\| \hat
    u_\eps(0) -  \hat u^0 \|^2_V %+ |\hat \E   \phi(\hat u_\eps(0)) - \hat
    % \E\phi(u^0)|
    \Bigg)\nonumber\\
&\quad \stackrel{\eqref{eq:equiv}}{=} \liminf_{\eps\to 0}  \hat I (\hat  \Omega,
\hat  \cF,(\hat  \cF_{ \eps, t})_{t\in[0,T]} ,  \hat  \P,   \hat   W_\epsi
                                                                  ,   \hat u^0_\eps,  \hat
  u_\epsi) \stackrel{\eqref{prev}}{=} \liminf_{\eps \to 0} I(u_\eps).\qedhere
  \end{align*}
\end{proof}

The combination of Propositions
\ref{prop:coercivity} and \ref{prop:lsc} still does not allow to prove the
existence of minimizers of the EDP functional $I$,
for the stochastic basis is changed in the limit. In the special case of an atomic
$\Omega$, however, no change in the basis is actually required and one
can find a minimizer of $I$ via the Direct Method. 

We conclude this section by proving that minimizers $u$ of $I$ are
actually null-minimizers ($I(u)=0$), 
hence solve \eqref{eq:00}. The reader is referred to
\cite{Auchmuty93,Ghoussoub08,Roubicek00} for some similar argument,
although in different variational settings. 

\begin{proposition}[Minimizers are null-minimizers]\label{prop:null} Assume that $F(t,\cdot)$ and
  $G(t,\cdot) $ are Gateaux-differentiable for all $t \in [0,T]$ and
  let $u\in \cU$ minimize $I$. Then, $I(u)=0$. 
\end{proposition}

\begin{proof}
  In order to check that $I(u)=0$ let us start by considering the linear problem
\begin{align}
  \d v + D_{\mathcal G} \partial \phi( u) v \, \d t = \big(D_{\mathcal
  G}  F(\cdot,  u) v -f\big)\, \d t + \big(D_{\mathcal
  G}  G(\cdot,  u) v - g\big) \, \d  W, \quad v(0)=z, \label{eq:v}
\end{align}
where $f \in L^2_{ \cP}( \Omega;L^2(0,T;H))$, $g \in L^2_{
  \cP}( \Omega;L^2(0,T;\cL^2(U,V)))$, and $z \in L^2(\Omega,
 \cF_0;V)$ are given. Owing to our assumptions, we have that the
latter is uniquely solvable, for the time-dependent positive linear operator $ D_{\mathcal
  G} \partial \phi( u)$ is coercive, uniformly with respect to
time. 

We now use the equivalence of Proposition
\ref{prop:equiv} in order to rewrite
\begin{align}
    &\hat I (u)=
      \frac12  \E \int_0^T \| \partial_t u^d +\partial \phi(u) -
      F(\cdot,u)\|^2 \, \ds\nonumber\\
  &\qquad+ 2C_\phi \E \int_0^T \| u^s -
    G(\cdot,u)\|^2_{\cL^2(U,V)} \, \ds +  \E\| u(0) -  u^0\|^2_V
    .%+ |\hat\E \phi(u(0)) - \hat\E\phi(u^0)|.
    \label{eq:equivd}
  \end{align} 
Let now $v$ be the solution of \eqref{eq:v}  and  compute the
variation of $ I$ at $ u$ in direction $v$
by letting $0=g'(0)$ for $g(t) = I( u + t
v)$. Owing to the Gateaux differentiability of $\partial \phi$, $F(t,\cdot)$, and
$G(t,\cdot) $ we obtain that
\begin{align*}
  0 &=  \E \int_0^T\big(\partial_t  u^d + \partial \phi( u) -
  F(\cdot,  u), \partial_t v^d +D_{\mathcal G}\partial \phi(
  u)v -D_{\mathcal G}F(\cdot, u)v \big)\, \d s\\
  &+ \E \int_0^T \big( u^s - G(\cdot,  u), v^s -  D_{\mathcal G}G(\cdot, u)v \big)_{\cL^2(U,V)}\, \d s  + 2 \E (
    u (0) -  u^0, v(0))_V\\
&=  \E \int_0^T\big(\partial_t  u^d + \partial \phi( u) -
  F(\cdot,  u), f \big)\, \d s\\
  &+ \E \int_0^T \big( u^s - G(\cdot,  u), g\big)_{\cL^2(U,V)}\, \d s  + 2 \E (
    u (0) -  u^0, z)_V.
\end{align*}
Since $f$, $g$, and $z$ are arbitrary we have proved that $\hat u$
solves $\partial_t  u^d + \partial \phi( u) = 
  F(\cdot,  u)$, $\hat u^s = G(\cdot,  u)$, and $
  u(0)= u^0$ a.e. Hence, $ u $solves \eqref{eq:00}. In particular,
$\hat I  ( u)=0$.
% On the other hand, one can
% easily check that martingale solutions of  \eqref{eq:00} are pathwise
% unique. A standard argument hence ensures that the convergence of $u_\eps$ 
% also holds in the original stochastic basis $(\Omega, \cF,(\cF_t)_{t\in
%   [0,T]},\P)$, without the need of passing to a different
% basis. In particular, $u$ and solves \eqref{eq:00}. Hence, $I(u)=0$. Eventually, since solutions of
% \eqref{eq:00} are unique, no extraction of subsequences is actually
% needed.  
\end{proof}

%%%%%%%%%%%%%%%%%%%%%%%%%%%%%%%%%%%%%%%%%%%%%%%%
\section{Application to stability}
\label{sec:stability}

Let us now give an application of Theorem \ref{thm:main} to the
analysis of the
stability of problem \eqref{eq:00} with respect to data perturbations. In the deterministic case, such stability results have to be
traced back to Attouch \cite{Attouch}. See also \cite{Visintin} for some
recent developments. In the stochastic regime, the reader is referred to
Gess \& T\"olle \cite{Gess}, where the case $\phi_n \to \phi$  is 
discussed. 

Assume to be given a sequence
$(u^0_\eps,\phi_\eps,F_\eps,G_\eps)_{\eps>0}$ of data, as well as a
limiting data set $(u^0_0,\phi_0,F_0,G_0)$, all fulfilling the assumptions
of Section \ref{sec:setting}, uniformly with respect to $\eps\in[0,1)$.  
We are interested in qualifying the convergences $u^0_\eps\to u^0_0$, $\phi_\eps \to \phi_0$, $F_\eps \to F_0$, and
$G_\eps \to G_0$ in such a way that solutions $u_\eps$ of equation
\eqref{eq:00} with data $(u^0_\eps,\phi_\eps,F_\eps,G_\eps) $, namely,
\begin{equation}
  \label{eq:00e}
   u_\eps(t) =  u^0_\eps + \int_0^t (F_\eps(\cdot,u_\eps)-\partial \phi_\eps(u_\eps)) \, \d s + \int_0^t
  G_\eps(\cdot,u_\eps)\, \d W\quad \forall t \in [0,T], \ \P\text{-a.s.}
\end{equation}
converge to solutions $u_0$ of equation
\eqref{eq:00} with data $(u^0_0,\phi_0, F_0,G_0) $, that is
\begin{equation}
  \label{eq:000}
   u_0(t) =  u^0_0 + \int_0^t (F_0(\cdot,u_0)-\partial \phi_0(u_0)) \, \d s + \int_0^t
  G_0(\cdot,u_0)\, \d W\quad \forall t \in [0,T], \ \P\text{-a.s.}
\end{equation}
The reformulation of these problems in terms of null-minimization of
EDP functionals allows to readily treat the stability question. In the case of gradient flows, the
approach dates back to Sandier \& Serfaty
\cite{Sandier2004,Serfaty11}. Recently, this variational treatment of limiting
processes has been applied to different kind of
parameter-dependent nonlinear
dissipative evolution problems and has been originating the concept of
{\it EDP convergence} \cite{Dondl19,Frenzel,Mielke20,Mielke21}. To the
best of our knowledge, we present here the first application of this
technique in the stochastic setting.

Let $I_\epsi$ and $I_0$ indicate the EDP
functionals \eqref{eq:EDP} defined with  data
$(u^0_\eps, \phi_\eps,F_\eps,G_\eps)_{\eps>0}$ and $(u^0_0,
\phi_0,F_0,G_0)$, respectively. In order to prove that $u_0$ solves
\eqref{eq:000} one has to check that $I_0(u_0)=0$. Since $I_0$ is
nonnegative, this would follow from 
\begin{equation}
I_0(u_0) \leq  \liminf_{\eps \to 0}
I_\eps(u_\eps) = 0.
\end{equation}
This is nothing but a $\Gamma$-$\liminf$
inequality for 
$I_\epsi$ \cite{DalMaso}, which we check below, by extending the
argument of Proposition \ref{prop:lsc}. In fact, the EDP functional
approach is flexible enough to deliver  convergence also for {\it approximate}
minimizers $v_\eps$ of $I_\eps $, namely for $I_\eps(v_\eps) \to 0$.
The main result of this section is the following. 

\begin{theorem}[Stability]\label{thm:stab} Let   
$(u^0_\eps,\phi_\eps,F_\eps,G_\eps)_{\eps>0}$ and
$(u^0_0,\phi_0,F_0,G_0)$ fulfill the assumptions
of Section \ref{sec:setting}, uniformly with respect to
$\eps$. Moreover, assume that, as $\epsi \to 0$,
\begin{align}
  &u^0_\eps \to u^0_0 \quad \text{in} \ L^2(\Omega, \cF_0;V), % \ \  \text{and} \ \  \E\phi_\epsi(u_\epsi^0) \to                 \E\phi_0(u^0_0),
    \label{stab1}\\
    & \phi_\eps \to \phi \quad \text{in  the  Mosco sense in} \ H, \label{stab2}
\end{align}
and that  for all $ w_\eps \wstarto w_0$ in  $\cV$ the following
convergences hold
\begin{align}
  &F_\eps(\cdot ,w_\eps) \weakto F_0(\cdot,w_0) \quad \text{in} \
    L^2_\cP(\Omega;L^2(0,T;H)), \label{stab3}\\
   &G_\eps(\cdot ,w_\eps) \weakto G_0(\cdot,w_0) \quad \text{in} \
    L^2_\cP(\Omega;L^2(0,T;\cL^2(U;V))).  \label{stab4}% \\
  % & {\rm Tr} (w^s_\eps (w^s_\eps)^*D_{\mathcal G} \partial \phi_\eps(w_\eps))\to  {\rm Tr} (w^s_0 (w^s_0)^*D_{\mathcal G} \partial \phi_0(w_0))\quad \text{in} \
  %   L^1_\cP(\Omega;L^1(0,T))  \label{stab5}
\end{align}
 If $I_\eps(v_\eps)\to 0$ than  $v_\eps \wstarto u_0$
in $\cV$, where $u_0$ solves \eqref{eq:000}.
\end{theorem}

\begin{proof}
  As $I_\eps(v_\eps)\to 0$, the sequence $(v_\eps)_\eps$ is bounded in
  $\cV$ by Proposition \ref{prop:coercivity}. This implies that $\partial_t v_\eps^d$ and $F(\cdot,v_\eps)$
  are bounded in $L^2_\cP(\Omega;L^2(0,T;H))$ and $v^s_\eps$ and
  $G(\cdot, v^s_\epsi)$ are bounded in
  $L^2_\cP(\Omega;L^2(0,T;\cL^2(U,V)))$. Moreover, since
  $I_\epsi(v_\eps)$ are bounded we have that $\partial
  \phi_\epsi(v_\epsi)$ are bounded in $L^2_\cP(\Omega;L^2(0,T;H))$ as
  well.

  Define $\iota_\eps := \int_0^\cdot v^s_\eps \, \d W$. 
By adapting the argument of Proposition \eqref{prop:lsc},   possibly
passing to a not relabeled subsequence we find a 
probability space $(\hat \Omega, \hat \cF, \hat \P)$, a
  sequence of measurable maps $\eta_\eps: (\hat \Omega, \hat \cF) \to
   (\Omega, \cF)$ with $\P \circ \eta_\epsi = \hat \P$ 
    for all $\eps>0$, and some measurable
    
   \[
   (\hat u_0, \hat u_0^0,\hat u_0^d, \hat \iota, \hat W): (\hat \Omega, \hat \cF) \to
   L^2(0,T; H)\times V\times C([0,T]; V^*)\times L^2(0,T; H) \times C([0,T];U_1),
   \]
   
such that, letting $\hat v_\eps := v_\eps \circ \eta_\eps$,
 $\hat \iota_\eps := \iota_\eps \circ \eta_\eps$, and
  $\hat W_\eps := W_\eps \circ \eta_\eps$, and the following
  convergences hold
  
  \begin{align}
    \hat v_\eps \wstarto  \hat u_0 \quad& \text{in} \  C([0,T];L^2(\hat
                                          \Omega,\hat \cF;V)) ,\label{conv:0v}\\ 
\hat v_\eps \to \hat u_0 \quad& \text{in} \ L^2(0,T;H), \
                              \hat\P\text{-a.s.},\label{conv:1v}\\ 
 \hat v_\eps^d \to \hat u_0^d \quad& \text{in} \ C([0,T];V^*), \
                              \hat\P\text{-a.s.},\label{conv:1v_bis}\\ 
\partial_t\hat v^d_\eps \weakto \partial_t \hat u^d_0 \quad& \text{in} \   
L^2(\hat\Omega;L^2(0,T;H)),
    \label{conv:3v} \\
     \hat v^s_\eps \weakto \hat u^s_0 \quad& \text{in} \   L^2
                                           (\hat\Omega;L^2(0,T;\cL^2(U,V)), \label{conv:4v}\\
 \hat v_\eps(0) \weakto \hat u_0^0 \quad& \text{in} \   L^2(\hat
                                         \Omega,\hat \cF_0;V), \label{conv:5v}\\
     \partial \phi_\eps(\hat v_\eps) \weakto \hat \xi_0 \quad& \text{in} \
                                                       L^2 (\hat\Omega;L^2(0,T;H)). \label{conv:6v}
  \end{align}
  
 % Owing to % convergence \eqref{conv:5v} and the Mosco convergence
 %  \eqref{stab2} we obtain
 %  \begin{equation}
 %    \label{eq:phi0}
 %   \hat \E \phi_0(\hat u_0^0) \leq \liminf_{\epsi \to 0} \hat \E \phi_\epsi(\hat v_\epsi(0)).
 % \end{equation}
 % Moreover, again
 The Mosco convergence \eqref{stab2} 
 together with convergences \eqref{conv:1v} and
  \eqref{conv:6v} ensures that $ \hat \xi_0  = \partial \phi_0(\hat u_0)$
  a.e., hence
  \begin{equation}
    \label{eq:phi1}
    \partial \phi_\eps(\hat v_\eps) \weakto  \partial \phi_0(\hat
    u_0)  \quad  \text{in} \
                                                       L^2(\hat\Omega;L^2(0,T;H)). 
  \end{equation}
  Eventually, the weak-continuous-convergence properties
  \eqref{stab3}-\eqref{stab4} entail that
  \begin{align}
  &F_\eps(\cdot ,\hat v_\eps) \weakto F_0(\cdot,\hat u_0) \quad \text{in} \
    L^2(\hat \Omega;L^2(0,T;H)), \label{eq:f3}\\
   &G_\eps(\cdot ,\hat v_\eps) \weakto G_0(\cdot,\hat u_0) \quad \text{in} \
    L^2(\hat \Omega;L^2(0,T;\cL^2(U;V))).  \label{eq:g3} 
  \end{align}
  
  We set now $(\hat\cF_{\eps,t})_{t\in[0,T]}$ as the filtration 
   generated by $(\hat v_\eps,\hat v_\eps^d, \hat \iota_\eps,\hat W_\eps)$, 
  and  $(\hat\cF_{t})_{t\in[0,T]}$ as the filtration 
   generated by $(\hat u_0,\hat u_0^d, \hat \iota,\hat W)$. As in the previous section, 
   a classical argument (see \cite{ScarStef-SDNL, biot}) ensures again that 
   $\hat W$ is a $U$-cylindrical Wiener process, $\hat \iota=\hat u^s\cdot\hat W$ and 
   \[
   \hat u_0 = \hat u_0^0 + \int_0^\cdot\partial_t\hat u_0^d(s)\,\d s +
   \int_0^\cdot \hat u_0^s(s)\,\d  \hat W (s).
  \]
  
  Using the equivalence \eqref{eq:equiv}, the  convergence of the 
  initial data \eqref{stab1},  and convergences
  \eqref{conv:3v}-\eqref{conv:5v} and
  \eqref{eq:f3}-\eqref{eq:g3} we can pass to the liminf in
  $I_\eps(v_\eps) = I_\eps (\hat  \Omega,
\hat  \cF,(\hat  \cF_{\eps,t})_{t\in[0,T]} ,  \hat  \P,  \hat   W,\hat v_\eps)$ and check
  that \begin{align*}
  &\hat I_0 (\hat  \Omega,
\hat  \cF,(\hat  \cF_{t})_{t\in[0,T]} ,  \hat  \P,  \hat   W,\hat u_0)
  \\
  &\quad \stackrel{\eqref{eq:equiv}}{=}
\frac12 \hat \E \int_0^T \| \partial_t \hat u^d_0
  +\partial \phi_0(\hat u_0) -
    F_0(\cdot,\hat u_0)\|^2 \, \d s\\
  &\qquad + 2C_\phi \hat \E \int_0^T \| \hat u^s_0 -
    G_0(\cdot,\hat u_0)\|^2_{\cL^2(U,V)} \, \d s  + \hat \E\| \hat
    u_0(0) - \hat u^0_0\|^2_V \\
 % &\qquad + |\hat \E \phi_0(\hat u_0(0)) - \hat \E\phi_0(u^0_0)| \\
  & \quad\hspace{1.3mm}\leq \liminf_{\eps\to 0}
\frac12 \hat \E \int_0^T \| \partial_t \hat v^d_\eps
  +\partial \phi_\eps(\hat v_\eps) -
    F_\eps(\cdot,\hat v_\eps)\|^2 \, \d s\\
  &\qquad + 2C_\phi \hat \E \int_0^T \| \hat v^s_\eps -
    G_\eps(\cdot,\hat v_\eps)\|^2_{\cL^2(U,V)} \, \d s  + \hat \E\|
    \hat v_\eps(0) - \hat u^0_\eps\|^2_V \\
%  &\qquad + |\hat \E  \phi_\eps(\hat v_\eps(0)) - \hat \E\phi_\eps(u^0_\eps)| \nonumber\\
&\quad \stackrel{\eqref{eq:equiv}}{=} \liminf_{\eps\to 0}  \hat I_\eps (\hat  \Omega,
\hat  \cF,(\hat  \cF_{\eps,t})_{t\in[0,T]} ,  \hat  \P,   \hat
                                                                          W_\eps ,\hat
  v_\epsi) = 0. 
       \end{align*}

   As $\hat I_0 (\hat  \Omega,
\hat  \cF,(\hat  \cF_t)_{t\in[0,T]} ,  \hat  \P,  \hat   W,\hat
u_0)=0$, Theorem \ref{thm:main} guarantees that $\hat u_0$ is a
{martingale solution} of \eqref{eq:00}.  As already
commented, the pathwise uniqueness of   martingale solutions of
\eqref{eq:00}   ensures that all the limits above
hold in the original stochastic basis $(\Omega, \cF,(\cF_t)_{t\in
  [0,T]},\P)$, as well, without the need of passing to a different
basis. In particular, the weak$\ast$ limit $u_0$ of $v_\eps$
fulfills $I_0(u_0) = 0$ and solves \eqref{eq:00}. Eventually, since solutions of
\eqref{eq:00} are unique, no extraction of subsequences is actually
needed.  
\end{proof}

%%%%%%%%%%%%%%%%%%%%%%%%%%%%%%%%%%%%%%%%%%%%%%%%
\section{Application to optimal control} 
\label{sec:control}

Consider now the equation
\begin{equation}\label{eq:0c}
  \d u + \partial \phi(u)\,\d t \ni f\,\d t + G(\cdot,u)\,\d W.
\end{equation}
This corresponds to equation to \eqref{eq:0}, where the nonlinearity
$F(\cdot,u)$ is replaced by $f \in L^2(0,T;H)$.  The datum $f$ is  interpreted as
a {\it control}, which for simplicity we 
assume to be deterministic. Given the initial value
$u^0$, the Cauchy problem for equation \eqref{eq:0c} corresponds to find a process $u\in \cU$ such that 
\begin{equation}
  \label{eq:00c}
  u(t) =  u^0 + \int_0^t (f-\partial \phi(u)) \, \d s + \int_0^t
  G(\cdot,u)\, \d W\quad \forall\,t \in [0,T], \ \P\text{-a.s.}
\end{equation}
Under  our  assumptions, for all $f \in L^2(0,T;H)$ there exists
a unique  solution  $u\in \cU\cap
L^2_{\cP}(\Omega;L^\infty(0,T;V))$  of 
\eqref{eq:00c}. This defines the {\it solution operator}
$$S: L^2(0,T;H) \to \cU\cap L^2_{\cP}(\Omega;L^\infty(0,T;V)), \quad
S(f):=u.$$
We are interested in the following optimal control problem
\begin{equation}
  \label{eq:oc}
  \min_{f \in \cA}\{J(f,u) \, : \, u = S(f) \}.
\end{equation}
Here, $\cA$ represent the set of {\it admissible} controls, which we
assume to be nonempty and weakly compact in $L^2(0,T;H)$, and $J:
L^2(0,T;H)\times \cV \to [0,\infty)$ is an abstract {\it target functional},
here considered to be lower semicontinuous with respect to the
weak$\ast$ topology in $L^2(0,T;H)\times \cV$.   An  $f^*$ solving
\eqref{eq:oc} is called {\it
  optimal control}, the corresponding $u^*=S(f)$ is an  {\it optimal state},
and the pair $(f^*,u^*)$ is  an  {\it optimal pair}. Let us start from the
following. 

\begin{proposition}[Existence] There exists an optimal pair
  $(f^*,u^*)$ for problem \eqref{eq:oc}.
\end{proposition}

\begin{proof}
  Let $f_n \in \cA$ be a infimizing sequence for
  problem \eqref{eq:oc}. As $\cA$ is weakly compact in $L^2(0,T;H)$ we can
  extract a not relabeled subsequence such that $f_n \weakto f^*$ in
  $L^2(0,T;H)$. By letting $u_n :=S(f_n)$ and using Theorem
  \ref{thm:stab} we have that $u_n \wstarto u^*$ in $\cV$, where $u^*
  \in S( f^*)$. Owing to the lower semicontinuity of $J$ we get
  $$J(f^*,u^*) \leq \liminf_{n\to \infty} J(f_n,u_n) = \inf_{f\in A}
  \{J(f,u)\, : \, u = S(f)\}$$
  so that $(f^*,u^*)$ is an optimal pair.
\end{proof}

By the characterization of Theorem \ref{thm:main}, one readily finds
that 
$$u = S(f) \ \ \Leftrightarrow \ \ I(f,u) =0,$$
where the {\it controlled EDP
  functional}   $I:
L^2(0,T;H)\times \cU \to [0,\infty]$
  is defined as
\begin{align}
 I (f,u) &= 
          \E\phi(u(T)) - \E \phi(u(0)) + \frac12 \E \int_0^T \|
  \partial_t u^d\|^2\, \ds   + \frac12 \E \int_0^T \| \partial
    \phi(u)-f  \|^2 \, \ds  \nonumber \\
  &\quad - \E \int_0^T
  (\partial_t u^d, f) \, \ds - \frac12 \E \int_0^T {\rm Tr}_H\, L(u)\, \ds 
     \nonumber \\
  &\quad + 2C_\phi\E\int_0^T \| u^s -
    G(\cdot,u)\|^2_{\cL^2(U,V)} \, \ds   +\E\|u(0)- u^0\|^2_V %\nonumber %\\[2mm]  &\quad + |\E \phi(u(0)) - \E \phi(u^0)|
\label{eq:EDPc}
\end{align}
if $u \in C([0,T];L^2(\Omega,\cF;V))$  and $I(u)=\infty$
otherwise. The controlled EDP functional can be used to penalize the  
SPDE constraint $u = S(f)$  in problem \eqref{eq:oc}.  We consider the {\it penalized optimal control} problems
\begin{equation}
  \label{eq:ocd}
  \min_{ \cA\times \cU} F_\delta \quad \text{with}  \quad   F_\delta(f,u) := J(f,u) + \frac{1}{\delta}  I(f,u)  
\end{equation}
where $\delta>0 $ is the penalization parameter. Let us mention that the penalization of optimal
control problems via weighted residuals is classical and can be traced
back to Lions \cite{Lions68}. Indeed, it has already been applied to different
stationary and evolutive situations, see 
\cite{Bergounioux92,Bergounioux94,Bergounioux98,Gariboldi09,Gunzburger00,Mophou11}
for a collection of results.  In the
deterministic  setting,  this penalization method via EDP functionals has been discussed in
\cite{portinale}. 

By combining the coercivity and the lower-limit tool from Propositions
\ref{prop:coercivity}-\ref{prop:lsc} one can find a minimizer of
$F_\delta$ for each fixed $\delta>0$, at the price of possibly
changing the underlying stochastic basis.

In the limit $\delta \to 0$ one recovers optimal pairs for the
 original problem \eqref{eq:oc} as limit of {\it approximate optimal pairs}
 at level $\delta$, without redefining the
 stochastic basis. The main result of this section is the following.

 \begin{theorem}[Limit $\delta\to 0$] Let $(f_\delta,u_\delta)_{\delta>0} \in
   \cA\times \cU$ be  such that
   \begin{equation}
     \liminf_{\delta \to 0}
   \left(F_\delta (f_\delta,u_\delta) - \inf_{ \cA\times \cU} F_\delta\right) =0.\label{eq:inf}
 \end{equation} Then, up to a
   not relabeled subsequence we have that 
   $(f_\delta,u_\delta)  \wstarto (f^*,u^*)$ in $L^2(0,T;H)\times \cV$
   where $(f^*,u^*)$ is an optimal pair for \eqref{eq:oc}.
 \end{theorem}

 \begin{proof}
   Choose $f_0\in \cA$ and let $u_0 = S(f_0)$, so that
   $I(f_0,u_0)=0$. Owing to the weak compactness of $\cA$ into
   $L^2(0,T;H)$ we can extract without relabeling so that $f_\delta \wto
   f $ in  $L^2(0,T;H)$.
   From \eqref{eq:inf} 
   we get that
   $$\liminf_{\delta \to 0} \frac{1}{\delta} I(f_\delta,u_\delta) \leq
   \liminf_{\delta \to 0}F_\delta (f_\delta,u_\delta)  \leq
   \limsup_{\delta \to 0} \inf_{ \cA\times \cU} F_\delta \leq J(f_0,u_0).$$
 Again by extracting some not relabeled subsequence, this entails that
   $$\frac{1}{\delta} I (f_\delta,u_\delta) \leq  1+J(f_0,u_0).$$ 
   As $\delta \to 0$ one has that $I (f_\delta,u_\delta) \to 0$ and we
   are in the setting of Theorem \ref{thm:stab}. In particular,
   $u_\delta \wstarto u^*$ in $\cV$ and $u^* =S(f^*)$. Moreover, owing
   to the lower semicontinuity of $J$, for any $u\in S(f)$ we find
   $$J(f^*,u^*)\leq \liminf_{\delta \to 0} J(f_\delta,u_\delta) \leq
   \liminf_{\delta \to 0} F_\delta (f_\delta,u_\delta)  \leq
   \limsup_{\delta \to 0}  \inf_{ \cA\times \cU}  F_\delta \leq J(f,u) $$
   which proves that the $(f^*,u^*)$ is an optimal pair for problem
   \eqref{eq:oc}.  
 \end{proof}

%%%%%%%%%%%%%%%%%%%%%%%%%%%%%%%%%%%%%%

\section*{Acknowledgement}
LS  is partially supported by the Austrian Science Fund (FWF) through the
Lise-Meitner project M\,2876. US is partially supported by 
the Austrian Science Fund (FWF) through projects F\,65, W\,1245,  I\,4354, I\,5149, and P\,32788, and by the OeAD-WTZ project CZ
01/2021.

%%%%%%%%%%%%%%%%%%%%%%%%%%%%%%%%%%%%%%

%\bibliography{ref}{}
\bibliographystyle{abbrv}

%\end{document}

\end{document}